\documentclass[a4paper,12pt]{article}
\usepackage{amssymb,amsmath,amsfonts,pifont}

\begin{document}
\newenvironment{proof}
               {\begin{sloppypar} \noindent{\bf Proof.}}
               {\hspace*{\fill} $\square$ \end{sloppypar}}
\newtheorem{atheorem}{\bf \temp}[section]
\newtheorem{thm}[atheorem]{Theorem}
\newtheorem{cor}[atheorem]{Corollary}
\newtheorem{lem}[atheorem]{Lemma}
\newtheorem{pro}[atheorem]{Property}
\newtheorem{prop}[atheorem]{Proposition}
\newtheorem{de}[atheorem]{Definition}
\newtheorem{rem}[atheorem]{Remark}
\newtheorem{fac}[atheorem]{Fact}
\newtheorem{ex}[atheorem]{Example}
\newtheorem{pr}[atheorem]{Problem}
\newtheorem{cla}[atheorem]{Assertion}

\title{Augmentation Quotients for Burnside Rings of some Finite $p$-Groups\footnote{Supported
by the NSFC (No. 11401155).}}
\author{\small Shan Chang \quad changshan@hfut.edu.cn}
\date{}
\maketitle
\vspace{-1cm}
\begin{center}
{\small School of Mathematics, Hefei University of Technology, Hefei 230009, China}
\end{center}

\begin{abstract}
Let $G$ be a finite group, $\Omega(G)$ be its Burnside ring, and $\Delta(G)$ its augmentation ideal.
Denote by $\Delta^n(G)$ and $Q_n(G)$ the $n$-th power of $\Delta(G)$ and the $n$-th consecutive quotient
group $\Delta^n(G)/\Delta^{n+1}(G)$, respectively. This paper provides an explicit $\mathbb{Z}$-basis
for $\Delta^n(\mathcal{H})$ and determine the isomorphism class of $Q_n(\mathcal{H})$ for each positive
integer $n$, where $\mathcal{H}=\langle g,h |\, g^{p^m}=h^p=1, h^{-1}gh=g^{p^{m-1}+1}\rangle$,
$p$ is an odd prime. 
\end{abstract}

\section{Introduction}
Let $G$ be a finite group. A $G$-{\it set} is a finite set $X$ together with an action of
$G$ on $X$:
\begin{equation}
G\times X \longrightarrow X, \quad (g,x)\mapsto gx, \quad g\in G, x\in X.
\end{equation}
Two $G$-sets $X$ and $Y$ are said to be {\it isomrphic} (denoted by $X \cong Y$), if there
exists a bijective map $f: X\longrightarrow Y$ such that
\begin{equation}\label{ee}
f(gx)=gf(x), \quad \forall\, g\in G,\, x\in X.
\end{equation}
It is easy to verify that isomorphism of $G$-sets is an equivalence relation. The equivalence classes
are called isomorphism classes. The isomorphism class of $X$ is denoted by $[X]$. The sum $[X]+[Y]$
of two isomorphism classes $[X]$ and $[Y]$
is defined by
\begin{equation}
[X]+[Y]=[X\sqcup Y],
\end{equation}
where $X\sqcup Y$ is the disjoint union of $X$ and $Y$, which is also a $G$-set in the canonical way.

The {\it Burnside ring} $\Omega(G)$ is the group completion of the monoid (under sum) of isomorphism
classes of $G$-sets. Its multiplication is induced by cartesian product of finite sets. Note that the
cartesian product $X\times Y$ is a $G$-set via coordinating action. By \cite{M}, $\Omega(G)$ is a
commutative ring with an identity element. Its underlying group is a finitely generated free abelian
group with basis the isomorphism classes of transitive $G$-sets. Hence, by \cite{M}, its free rank is
equal to the number of conjugacy classes of subgroups of $G$.

The number of fixed points of $G$-set induces a ring homomorphism
\begin{equation}
\phi: \Omega(G) \longrightarrow \mathbb{Z}.
\end{equation}
This homomorphism is called the {\it augmentation map}. Its kernel $\Delta(G)$ is called the
{\it augmentation ideal} of $\Omega(G)$. Denote by $\Delta^n(G)$ the $n$-th power of $\Delta(G)$.
The $n$-th {\it augmentation quotient group} of $\Omega(G)$ is defined as
\[Q_{n}(G)=\Delta^n(G)/\Delta^{n+1}(G).\]

The problem of determining the structure of $\Delta^n(G)$ and $Q_{n}(G)$ is an interesting topic
in group ring theory. In \cite{WT}, Haibo Wu and Guoping Tang determined the isomorphism class of
$Q_n(G)$ for all finite abelian groups and for any positive integer $n$. Gaohua Tang, Yu Li and
Yansheng Wu settled the problem completely for certain nonabelian 2-groups in \cite{TLW}. In \cite{C0},
Shan Chang provided an explicit $\mathbb{Z}$-basis of $\Delta^n(\mathcal{D})$ and determined the
isomorphism class of $Q_n(\mathcal{D})$ for each positive integer n, where $\mathcal{D}$ is the
generalized dihedral group of a finite abelian group of odd order. However, the isomorphism class
of $Q_{n}(G)$ remained unclear for other non-abelian groups.

Let $p$ be a fixed odd prime. For each positive integer $m$ greater than or equal to 2, it is
well known that there is exactly one isomorphism class of nonabelian $p$-group of order $p^{m+1}$
which have a cyclic subgroup of index $p$. Denote this nonabelian $p$-group by $\mathcal{H}$.
Its presentation is
\begin{equation}
\mathcal{H}=\langle a,b |\, a^{p^m}=b^p=1, b^{-1}ab=a^{p^{m-1}+1}\rangle.
\end{equation}
The goal of this article is to provide an explicit $\mathbb{Z}$-basis for $\Delta^n(\mathcal{H})$
and determine the isomorphism class of $Q_{n}(\mathcal{H})$ for each positive integer $n$. Yanan Wen
and Shan Chang settled this problem for $m=2$ in \cite{WC}, thus we assume $m\geqslant 3$ in the sequel.

The result also yields $\mbox{Tor}^{\Omega(\mathcal{H})}_1(\Omega(\mathcal{H})/\Delta^n(\mathcal{H}),
\Omega(\mathcal{H})/\Delta(\mathcal{H}))$ because for any finite group $G$, $Q_n(G) \cong
\mbox{Tor}^{\Omega(G)}_1(\Omega(G)/\Delta^n(G), \Omega(G)/\Delta(G))$.

Two related problems of recent interest have been to investigate the augmentation ideals and their
consecutive quotients for integral group rings and representation rings of finite groups. These
problems have been well studied in \cite{P}, \cite{T1}, \cite{T2}, \cite{BT}, \cite{T3}, \cite{CT},
\cite{CCT}, \cite{C1} and \cite{C2}.

\section{Preliminaries}
In this section, we provide some useful results about $\Omega(G)$, $\Delta^n(G)$, $Q_n(G)$ and
finitely generated free abelian groups. Indeed, lemma \ref{th11}, lemma \ref{th12}, theorem
\ref{th13} and corollary \ref{th14} were found in \cite{M}, lemma \ref{th2} and lemma \ref{th8}
were proved in \cite{C0}.

Let $X$ be a $G$-set. There is an equivalence relation on $X$ given by saying $x$ is related to
$y$ if there exists $g\in G$ with $gx=y$. The equivalence classes are called {\it orbits}. It is
easy to see that the orbits of $x$ is
\begin{equation}
Gx=\{gx \,|\, g\in G\}.
\end{equation}
A $G$-set is {\it transitive} if it has only one orbit. For instance, if $x\in X$, the orbit $Gx$
is a transitive $G$-set. As another example, if $K$ is a subgroup of $G$, the set $G/K$ of left
cosets of $K$ is a transitive $G$-set under $g(hK)=(gh)K$.

There is a standard form for each transitive $G$-set. The {\it stabilizer} of an element $x\in X$ is
the set
\begin{equation}
G_x=\{g\in G \,|\, gx=x\}.
\end{equation}
It is easy to verify that the stabilizer $G_x$ is a subgroup of $G$. The following lemma shows that
every transitive $G$-set is isomorphic to $G/K$ for some subgroup $K$ of $G$.
\begin{lem}\label{th11}
For any $G$-set $X$ and $x\in X$,
\begin{equation}
Gx\cong G/G_x, \quad gx\mapsto gG_x.
\end{equation}
\end{lem}

Note that $G_{gx}=gG_xg^{-1}$ for each $x\in X$ and $g\in G$. Moreover, any isomorphism
of $G$-sets $f:X\longrightarrow Y$ preserves stabilizers: $G_x=G_{f(x)}$. From these two facts we get
the following lemma.
\begin{lem}\label{th12}
Let $K$ and $L$ be two subgroups of $G$. Then $G/K\cong G/L$ if and only if $K$ and $L$ are conjugate
in $G$.
\end{lem}

Assembling above lemmas, we have proved the following theorem.
\begin{thm}\label{th13}
Let $\mathcal{K}$ be a set of subgroups of $G$, one chosen from each conjugacy class of subgroups of $G$.
Then $\Omega(G)$ is, additively, the free abelian based on $\big\{ [G/K] \,|\, K\in \mathcal{K} \big\}$.
\end{thm}

Recall that $\Delta(G)$ is the kernel of $\phi:\Omega(G)\longrightarrow\mathbb{Z}$ which sending
$[X]$ to the number $\#(X)$ of fixed points of $X$. Brief calculations show that
\begin{equation}
\#(G/K)=\left\{\begin{array}{ll}
1, & \mbox{if } K=G, \\
0, & \mbox{if } K<G.
\end{array}\right.
\end{equation}
From this we get the following corollary.
\begin{cor}\label{th14}
For each finite group $G$, the underlying group of $\Delta(G)$ is the free abelian group based on
$\big\{ [G/K] \,|\, K\in \mathcal{K}, K<G \big\}$.
\end{cor}

The multiplication in $\Omega(G)$ is completely determined by the product
\begin{equation}
[G/K][G/L]=[(G/K)\times (G/L)],
\end{equation}
where $K,L$ are subgroups of $G$. The following lemma tackles this product.
\begin{lem}\label{th2}
Let $K$ be a subgroup of $G$, $L$ be a normal subgroup of $G$. Then
\begin{equation}
[G/K][G/L]=\frac{|G|\cdot |K\cap L|}{|K|\cdot |L|}[G/(K\cap L)]=\frac{|G|}{|KL|}[G/(K\cap L)].
\end{equation}
\end{lem}

It is easy to see that $\Delta^n(G)$ is a finitely generated free abelian group for any positive
integer $n$. However, it is usually difficult to write down explicitly a basis of $\Delta^n(G)$,
even for a finite $p$-group. The following lemma sheds some light on its free rank.
\begin{lem}\label{th8}
For each positive integer $n$, $Q_n(G)$ is a finite abelian group, hence $\Delta^n(G)$ has same free
rank as $\Delta(G)$.
\end{lem}

At last, we recall a classical result about finitely generated free abelian groups.
\begin{lem}\label{th3}
Let $F$ be a finite generated free abelian group of rank $r$. If the $r$ elements $g_1,\ldots,g_r$
generate $F$, then they form a basis of $F$.
\end{lem}

\section{Necessary Tools}
In this section, we construct a basis of $\Delta(\mathcal{H})$ as a free abelian group. Then we
determine the multiplication in $\Delta(\mathcal{H})$.

Recall that the presentation of $\mathcal{H}$ is $\langle a,b |\, a^{p^m}=b^p=1, b^{-1}ab=a^{p^{m-1}+1}\rangle$.
The following lemma determines all proper subgroups of $\mathcal{H}$. For convenience, we fix the
following notation.
\begin{itemize}
\item For any positive integer $j$, set $\overline{j}=\{0,1,\ldots,j\}$, $\underline{j}=\{1,\ldots,j\}$.
\item Denote by $N_i$ the cyclic subgroup of $\mathcal{H}$ which generated by $a^{p^i}$, $i\in \overline{m}$.
\item For any subset $\Gamma\subset\Omega(\mathcal{H})$, denote by $\mathbb{Z}\Gamma$ the set of
all $\mathbb{Z}$-linear combinations of elements of $\Gamma$.
\item Denote by $C_p$ the cyclic group of order $p$.
\end{itemize}
\begin{lem}\label{th1}
Let $K$ be a proper subgroup of $\mathcal{H}$.
\begin{itemize}
\item[\ding{172}] If $K\subset N_0$, then $K=N_i$ for some integer $i\in \overline{m}$.
Moreover, $K$ is a normal subgroup of $\mathcal{H}$.
\item[\ding{173}] If $K\not\subset N_0$, then there are two integers $k,l$ with $k\in \underline{m}$,
$l\in \overline{p-1}$ such that
\begin{equation}
K=\bigcup_{j=0}^{p-1} (ba^{lp^{k-1}})^j N_k=\bigcup_{j=0}^{p-1} b^ja^{jlp^{k-1}}N_k.
\end{equation}
\end{itemize}
\end{lem}
\begin{proof}
\ding{172} It is a direct corollary of the presentation of $\mathcal{H}$. \ding{173} Note that $KN_0$
is a subgroup of $\mathcal{H}$ which contains $N_0$ properly. This implies $KN_0=\mathcal{H}$ since
$N_0$ is a maximal subgroup of $\mathcal{H}$. Due to \ding{172}, there exists an integer $k\in \overline{m}$
such that $K\cap N_0=N_k$. Then we have the following group isomorphism,
\begin{equation}
\begin{array}{ccccc}
K/N_k & \cong & KN_0/N_0=\mathcal{H}/N_0 & \cong & \langle b\rangle, \\
b^ua^vN_k & \mapsto & b^ua^vN_0=b^uN_0 & \mapsto & b^u,
\end{array}
\end{equation}
where $\langle b\rangle$ is the cyclic subgroup of $\mathcal{H}$ which generated by $b$. Let $ba^vN_k$ be the
inverse image of $b$ under above isomorphism. Then $K/N_k$ is a cyclic group of order $p$ which generated
by $ba^vN_k$. From this it follows that
\begin{equation}
K=\bigcup_{j=0}^{p-1} (ba^v N_k)^j=\bigcup_{j=0}^{p-1} (ba^v)^j N_k,
\end{equation}
where $(ba^vN_k)^p=(ba^v)^pN_k=N_k$, i.e. $(ba^v)^p$ lies in $N_k$. We claim $k\neq 0$, otherwise $N_0$ is
properly contained in $K$, which is a contradiction since $N_0$ is a maximal subgroup. Without loss of
generality, we can assume $v\in \overline{p^k-1}$ since $ba^vN_k$ depends only on the residue class of
$v$ modulo $p^k$. To finish the proof, we need the following assertion.
\begin{cla}\label{th5}
For any natural numbers $w$ and $j$, we have
\begin{equation}
(ba^w)^j=b^ja^{w\big[\frac{j(j-1)}{2}p^{m-1}+j\big]}.
\end{equation}
\end{cla}
\begin{proof}
We prove the assertion by induction on $j$. When $j=0$, it is trivial. If $j\geqslant 1$, assume the
assertion holds for $j-1$ and any natural number $w$. Note that $a^wb=ba^{w(p^{m-1}+1)}$. Hence,
\begin{eqnarray}
(ba^w)^j & = & b(a^wb)^{j-1}a^w=b(ba^{w(p^{m-1}+1)})^{j-1}a^w \nonumber\\
& = & bb^{j-1}a^{w(p^{m-1}+1)\big[\frac{(j-1)(j-2)}{2}p^{m-1}+j-1\big]}a^w \nonumber\\
& = & b^ja^{w\big[\big(\frac{(j-1)(j-2)}{2}+j-1\big)p^{m-1}+j-1\big]}a^v \nonumber\\
& = & b^ja^{w\big[\frac{j(j-1)}{2}p^{m-1}+j\big]},
\end{eqnarray}
as required.
\end{proof}

We return now to the proof of Lemma \ref{th1}. Due to Assertion \ref{th5}, we get
\begin{equation}
(ba^v)^p=b^pa^{v\big[\frac{p(p-1)}{2}p^{m-1}+p\big]}=a^{vp}.
\end{equation}
From this and the fact $(ba^v)^p$ lies in $N_k$ it follows that $v$ is a multiple of $p^{k-1}$.
Set $v=lp^{k-1}$. Then $l\in \overline{p-1}$. Brief calculations show that, for any natural number $j$,
\[(ba^{lp^{k-1}})^j N_k=b^ja^{lp^{k-1}\big[\frac{j(j-1)}{2}p^{m-1}+j\big]} N_k=b^ja^{jlp^{k-1}} N_k.\]
Thus the lemma is proved.
\end{proof}

Due to Theorem \ref{th13}, we need a set of representatives of all conjugacy classes of subgroups of
$\mathcal{H}$ to construct a basis of $\Omega(\mathcal{H})$. By Lemma \ref{th1}, each subgroup of
$\mathcal{H}$ contained in $N_0$ is normal. Set
\begin{equation}
M_{kl}=\bigcup_{j=0}^{p-1} (ba^{lp^{k-1}})^j N_k, \quad k\in \underline{m},l\in\mathbb{N}.
\end{equation}
It is easy to verify that $M_{kl}$ is a subgroup of $\mathcal{H}$. For later use, we remind that,
for a fixed integer $k\in \underline{m}$, $M_{kl}$ depends only on the residue class of $l$ modulo $p$.
The following lemma determines their conjugacy classes.
\begin{lem}\label{th6}
For any natural number $l$, we have,
\begin{itemize}
\item[\ding{172}] $M_{kl}$ is a normal subgroup of $\mathcal{H}$ if $k\in \underline{m-1}$,
\item[\ding{173}] $M_{ml}$ is conjugate to $M_{m0}$.
\end{itemize}
\end{lem}
\begin{proof}
\ding{172} Let $k\in \underline{m-1}$. Then short calculations show
\begin{eqnarray}
a^{-1}(ba^{lp^{k-1}})a & = & ba^{-(p^{m-1}+1)}a^{lp^{k-1}}a \nonumber\\
& = & ba^{lp^{k-1}-p^{m-1}} \nonumber\\
& = & ba^{lp^{k-1}}(a^{p^k})^{-p^{m-1-k}}, \\
b^{-1}(ba^{lp^{k-1}})b & = & ba^{lp^{k-1}(p^{m-1}+1)} \nonumber\\
& = & ba^{lp^{k-1}}(a^{p^k})^{lp^{m-2}}.
\end{eqnarray}
This implies both $a^{-1}(ba^{lp^{k-1}})a$ and $b^{-1}(ba^{lp^{k-1}})b$ belong to $M_{kl}$.
Thus $M_{kl}$ is a normal subgroup since $M_{kl}$ is generated by $N_k$ and $ba^{lp^{k-1}}$.
\ding{173} Note that $N_m$ is the trivial subgroup, so $M_{ml}$ is generated by $ba^{lp^{m-1}}$.
Brief calculations show
\begin{equation}\label{e2}
a^{-l}(ba^{lp^{m-1}})a^l=ba^{-l(p^{m-1}+1)}a^{lp^{m-1}}a^l=b,
\end{equation}
which implies $M_{ml}$ is conjugate to $M_{m0}$.
\end{proof}

Thanks to Lemma \ref{th1} and Lemma \ref{th6}, we get a basis of $\Omega(\mathcal{H})$, hence a basis of
$\Delta(\mathcal{H})$. For convenience, denote $[\mathcal{H}/N_i],[\mathcal{H}/M_{kl}]$
and $[\mathcal{H}/\mathcal{H}]$ by $\alpha_i,\beta_{kl}$ and $\varepsilon$, respectively, where
$i\in \overline{m}$, $k\in \underline{m}$, $l$ is a natural number.
\begin{thm}
The underlying group of $\Omega(\mathcal{H})$ is the free abelian group with basis
\begin{equation}\label{e3}
\{\alpha_i|i\in \overline{m}\}\cup\{\beta_{kl}|k\in \underline{m-1},l\in \overline{p-1}\}\cup\{\beta_{m0},\varepsilon\}.
\end{equation}
\end{thm}
\begin{proof}
It is easy to verify that each element of $\mathcal{H}$ has a unique expression as $b^ua^v$,
$u\in \overline{p-1}, v\in \overline{p^m-1}$. By this one can easily verify that
\begin{equation}
\{N_i|i\in \overline{m}\}\cup\{M_{kl}|k\in \underline{m-1},l\in \overline{p-1}\}\cup\{M_{m0},\mathcal{H}\}
\end{equation}
is a set of representatives of all conjugacy classes of subgroups of $\mathcal{H}$. Then the theorem follows from Theorem \ref{th13}.
\end{proof}
\begin{cor}\label{th7}
$\Delta(\mathcal{H})$ is, additively, the free abelian group based on
\begin{equation}\label{e6}
\{\alpha_i|i\in \overline{m}\}\cup\{\beta_{kl}|k\in \underline{m-1},l\in \overline{p-1}\}\cup\{\beta_{m0}\}.
\end{equation}
\end{cor}

Now we determine the multiplication in $\Delta(\mathcal{H})$.
\begin{lem}\label{th4}
For any integers $i,j\in \overline{m}$, $k\in\underline{m}$, $r\in \underline{m-1}$ and $l,s\in \overline{p-1}$, we have,
\begin{itemize}
\item[\ding{172}] $\alpha_i\alpha_j=p^{\min\{i,j\}+1}\alpha_{\max\{i,j\}}$,
\item[\ding{173}] $\alpha_i\beta_{kl}=p^{\min\{i,k\}}\alpha_{\max\{i,k\}}$,
\item[\ding{174}] $\beta_{kl}\beta_{rs}=\left\{
\begin{array}{ll}
p^r\beta_{kl}, & \mbox{if }\, \{k>r,s=0\} \mbox{ or } \{k=r,l=s\}, \\
p^{r-1}\alpha_k, & \mbox{if }\, \{k>r,s\in \underline{p-1}\} \mbox{ or } \{k=r, l\neq s\},
\end{array}\right.$
\item[\ding{175}] $\beta_{m0}^2=p^{m-1}\beta_{m0}+(p^{m-1}-p^{m-2})\alpha_m$.
\end{itemize}
\end{lem}
\begin{proof}
\ding{172}, \ding{173} and \ding{174} are direct corollaries of Lemma \ref{th2}. For \ding{175},
it is easy to verify that
\begin{equation}
\mathcal{H}/M_{m0}=\{a^vM_{m0}|v\in \overline{p^m-1}\}.
\end{equation}
Let $x=(a^vM_{m0},a^wM_{m0})\in \mathcal{H}/M_{m0} \times \mathcal{H}/M_{m0}$. From (\ref{e2}) it follows that
\begin{equation}
\mathcal{H}_x=(a^vM_{m0}a^{-v}) \cap (a^wM_{m0}a^{-w})=M_{mv}\cap M_{mw},
\end{equation}
where $\mathcal{H}_x$ is the stabilizer of $x$. Recall that, for any natural number $l$, $M_{ml}$ depends only
on the residue class of $l$ modulo $p$. This implies
\begin{equation}
\mathcal{H}_x=\left\{\begin{array}{ll}
M_{mv}, & \mbox{if } v\equiv w \,(\mbox{mod } p),\\
\{1\}=N_m, & \mbox{if } v\not\equiv w \,(\mbox{mod } p).
\end{array}\right.
\end{equation}
Thus there are exactly $p^{2m-1}$ elements of $\mathcal{H}/M_{m0} \times \mathcal{H}/M_{m0}$ whose orbits are
isomorphic to $\mathcal{H}/M_{m0}$, and the orbits of the rest elements are isomorphic to $\mathcal{H}/N_m$.
Note that the cardinalities of $\mathcal{H}/M_{m0}$ and $\mathcal{H}/N_m$ are $p^m$ and $p^{m+1}$, respectively.
Then the lemma follows.
\end{proof}

\section{Main Results}
In this section, we construct an explicit $\mathbb{Z}$-basis for $\Delta^n(\mathcal{H})$ and determine
the isomorphism class of $Q_n(\mathcal{H})$ for each positive integer $n$. For later use, we remind the readers
that for a fixed $k\in\underline{m}$, $\beta_{kl}$ depends only on the residue class of $l$ modulo $p$. Moreover,
$\beta_{ml}$ equals $\beta_{m0}$ for any natural number $l$.

\begin{thm}\label{th10}
For any positive integer $n$, $\Delta^{n+1}(\mathcal{H})$ is, additively, the free abelian group based on
\begin{equation}\label{e8}
\{p^n\alpha_0\}\cup\{p^{n-1}\alpha_i|i\in \underline{m}\}\cup
\{p^n\beta_{kl}|k\in \underline{m-1},l\in \overline{p-1}\}\cup\{p^n\beta_{m0}\}.
\end{equation}
\end{thm}
\begin{proof}
Note that (\ref{e8}) has the same cardinality as (\ref{e6}), so due to Lemma \ref{th8}, Lemma \ref{th3} and
Corollary \ref{th7}, we just need to show it generates $\Delta^n(\mathcal{H})$. We prove this by induction on $n$.
When $n=1$, brief calculations show
\begin{eqnarray}
\Delta^2(\mathcal{H}) & = & \Delta(\mathcal{H})\Delta(\mathcal{H}) \nonumber\\
& = & \mathbb{Z}\{\alpha_i\alpha_j | i,j\in\overline{m},i\leqslant j\}+
      \mathbb{Z}\{\alpha_i\beta_{kl} | i\in\overline{m},k\in\underline{m},l\in\overline{p-1}\}+ \nonumber\\
&   & \mathbb{Z}\{\beta_{kl}\beta_{rs} | k\in\underline{m},r\in\underline{m-1}, l,s\in\overline{p-1}, k\geqslant r\}+
      \mathbb{Z}\{\beta_{m0}^2\} \nonumber\\
& = & \mathbb{Z}\{p^{i+1}\alpha_j | i,j\in\overline{m},i\leqslant j\}+
      \mathbb{Z}\{p^{\min\{i,k\}}\alpha_{\max\{i,k\}} | i\in\overline{m},k\in\underline{m}\}+ \nonumber\\
&   & \mathbb{Z}\{p^r\beta_{kl} | k\in\underline{m},r\in\underline{m-1},l\in\overline{p-1},k\geqslant r\}+ \nonumber\\
&   & \mathbb{Z}\{p^{r-1}\alpha_k | k\in\underline{m},r\in\underline{m-1},k\geqslant r\}+ \nonumber\\
&   & \mathbb{Z}\{p^{m-1}\beta_{m0}+(p^{m-1}-p^{m-2})\alpha_m\} \nonumber\\
& = & \mathbb{Z}\{p\alpha_j | j\in\overline{m}\}+\mathbb{Z}\{\alpha_k | k\in\underline{m}\}+
      \mathbb{Z}\{p\beta_{kl} | k\in\underline{m}, l\in\overline{p-1}\}+ \nonumber \\
&   & \mathbb{Z}\{\alpha_k | k\in\underline{m}\}+\mathbb{Z}\{p^{m-1}\beta_{m0}+(p^{m-1}-p^{m-2})\beta_{m0}\} \nonumber \\
& = & \mathbb{Z}\{p\alpha_0\}+\mathbb{Z}\{\alpha_i | i\in\underline{m}\}+\mathbb{Z}\{p\beta_{kl} | k\in\underline{m-1},
      l\in\overline{p-1}\}+ \nonumber\\
&   & \mathbb{Z}\{p\beta_{m0}\},
\end{eqnarray}
as required. When $n\geqslant 2$, assume the theorem holds for $n-1$, which means the underlying group of
$\Delta^n(\mathcal{H})$ is the free abelian group with basis
\begin{eqnarray}
&& \{p^{n-1}\alpha_0\}\cup\{p^{n-2}\alpha_j|j\in \underline{m}\}\cup
   \{p^{n-1}\beta_{rs}|r\in \underline{m-1},s\in \overline{p-1}\}\cup \nonumber \\
&& \{p^{n-1}\beta_{m0}\}.
\end{eqnarray}
From this it follows that
\begin{eqnarray}
\Delta^{n+1}(\mathcal{H})
&=& \Delta(\mathcal{H})\Delta^n(\mathcal{H}) \nonumber\\
&=& \mathbb{Z}\{p^{n-1}\alpha_i\alpha_0|i\in\overline{m}\}+
    \mathbb{Z}\{p^{n-2}\alpha_i\alpha_j|i\in\overline{m},j\in \underline{m}\}+ \nonumber\\
& & \mathbb{Z}\{p^{n-1}\alpha_i\beta_{rs}|i\in\overline{m}, r\in \underline{m-1},s\in \overline{p-1}\}+ \nonumber\\
& & \mathbb{Z}\{p^{n-1}\alpha_i\beta_{m0}|i\in\overline{m}\}+
    \mathbb{Z}\{p^{n-1}\beta_{kl}\alpha_0|k\in\underline{m},l\in\overline{p-1}\}+ \nonumber\\
& & \mathbb{Z}\{p^{n-2}\beta_{kl}\alpha_j|j,k\in\underline{m},l\in\overline{p-1}\}+ \nonumber\\
& & \mathbb{Z}\{p^{n-1}\beta_{kl}\beta_{rs}|k\in\underline{m},r\in \underline{m-1},l,s\in \overline{p-1}\}+ \nonumber\\
& & \mathbb{Z}\{p^{n-1}\beta_{kl}\beta_{m0}|k\in\underline{m},l\in\overline{p-1}\} \nonumber\\
&=& \mathbb{Z}\{p^{n-1}\alpha_0^2\}+
    \mathbb{Z}\{p^{n-2}\alpha_i\alpha_j|i\in\overline{m},j\in \underline{m}\}+ \nonumber\\
& & \mathbb{Z}\{p^{n-1}\beta_{kl}\alpha_0|k\in\underline{m},l\in\overline{p-1}\}+ \nonumber\\
& & \mathbb{Z}\{p^{n-2}\beta_{kl}\alpha_j|j,k\in\underline{m},l\in\overline{p-1}\}+ \nonumber\\
& & \mathbb{Z}\{p^{n-1}\beta_{kl}\beta_{rs}|k\in\underline{m},r\in \underline{m-1},l,s\in \overline{p-1}\}+
    \mathbb{Z}\{p^{n-1}\beta_{m0}^2\} \nonumber\\
&=& \mathbb{Z}\{p^n\alpha_0\}+\mathbb{Z}\{p^{n+i-1}\alpha_j|i\in\overline{m},j\in\underline{m},i\leqslant j\}+ \nonumber\\
& & \mathbb{Z}\{p^{n-1}\alpha_k|k\in\underline{m}\}+
    \mathbb{Z}\{p^{n+k-2}\alpha_j|j,k\in\underline{m},j\geqslant k\}+ \nonumber\\
& & \mathbb{Z}\{p^{n+r-1}\beta_{kl} | k\in\underline{m},r\in\underline{m-1},l\in\overline{p-1},k\geqslant r\}+ \nonumber\\
& & \mathbb{Z}\{p^{n+r-2}\alpha_k | k\in\underline{m},r\in\underline{m-1},k\geqslant r\}+ \nonumber\\
& & \mathbb{Z}\{p^{n-1}(p^{m-1}\beta_{m0}+(p^{m-1}-p^{m-2})\alpha_m)\} \nonumber\\
&=& \mathbb{Z}\{p^n\alpha_0\}+\mathbb{Z}\{p^{n-1}\alpha_j|j\in\underline{m}\}+\mathbb{Z}\{p^{n-1}\alpha_k|k\in\underline{m}\}+ \nonumber\\
& & \mathbb{Z}\{p^{n-1}\alpha_j|j\in\underline{m}\}+\mathbb{Z}\{p^n\beta_{kl} | k\in\underline{m},l\in\overline{p-1}\}+ \nonumber\\
& & \mathbb{Z}\{p^{n-1}\alpha_k | k\in\underline{m}\}+\mathbb{Z}\{p^{m-2}(p^n\beta_{m0}+(p^n-p^{n-1})\alpha_m)\} \nonumber\\
&=& \mathbb{Z}\{p^n\alpha_0\}+\mathbb{Z}\{p^{n-1}\alpha_i|i\in\underline{m}\}+ \nonumber\\
& & \mathbb{Z}\{p^n\beta_{kl} | k\in\underline{m-1},l\in\overline{p-1}\}+\mathbb{Z}\{p^n\beta_{m0}\}.
\end{eqnarray}
Thus the theorem is proved.
\end{proof}
\begin{thm}
For any positive integer $n$,
\begin{equation}
Q_n(\mathcal{H})\cong \left\{
\begin{array}{ll}
(C_p)^{(m-1)p+2}, & n=1, \\
(C_p)^{(m-1)p+m+2}, & n\geqslant 2.
\end{array}\right.
\end{equation}
\end{thm}
\begin{proof}
It is a direct corollary of Theorem \ref{th10}.
\end{proof}

\end{document}